\documentclass[leqno,12pt]{article}
\usepackage{amsmath, amssymb,soul}
\usepackage{amsthm} %theorem environment option

\theoremstyle{plain} %text of this environment is typesetted in italics
\newtheorem{theorem}{\indent\sc Theorem}[section]
\newtheorem{lemma}[theorem]{\indent\sc Lemma}
\newtheorem{corollary}[theorem]{\indent\sc Corollary}
\newtheorem{proposition}[theorem]{\indent\sc Proposition}

\numberwithin{equation}{section}
\theoremstyle{definition} %text of this environment is typesetted in roman letters
\newtheorem{definition}[theorem]{\indent\sc Definition}
\newtheorem{remark}[theorem]{\indent\sc Remark}

\newcommand\on{\operatorname}
\renewcommand\div{\on{div}}
\newcommand\grad{\on{grad}}
\newcommand\Hess{\on{Hess}}
\newcommand\Ric{\on{Ric}}
\newcommand\scal{\on{scal}}


\def\({\left( }
\def\){\right)}
\def\<{\left< }
\def\>{\right>}

\def\e{\eqref}

\def\x{{\bf x}}

\title{Harmonic forms and generalized solitons}
\author{Adara M. Blaga and Bang-Yen Chen}
\date{}

\begin{document}

\maketitle

\markboth{{\small\it {\hspace{4cm} Harmonic forms and generalized solitons}}}{\small\it{Harmonic forms and generalized solitons
\hspace{4cm}}}

%\begin{abstract}
%In this paper, we define the notion of a generalized soliton $(g,\xi,\eta,\beta,\gamma,\delta)$ on a manifold.
\textbf{Abstract:} For a generalized soliton $(g,\xi,\eta,\beta,\gamma,\delta)$ we provide necessary and sufficient conditions for the dual $1$-form $\xi^{\flat}$ of the potential vector field $\xi$ to be a solution of the Schr\"{o}dinger-Ricci equation, a harmonic or a Schr\"{o}dinger-Ricci harmonic form. We also characterize the $1$-forms orthogonal to $\xi^{\flat}$, underlying the results obtained for the Ricci and Yamabe solitons. Further, we formulate the results for the case of gradient generalized solitons.
Several applications and examples are also presented.\\

{ %2010 MSC numbers
}

{ %key words and phrases
}

%\end{abstract}

\section{Introduction}

The soliton phenomenon was firstly described by John Scott Russell (1808--1882) in 1834, who observed a solitary wave, that is, a wave keeping the shape while it propagates with a constant speed (see, e.g., \cite{A98}).
The notion of soliton as a stationary solution of a geometric flow on a Riemannian manifold has been lately considered and intensively studied.
In the 1980s, R.S. Hamilton introduced the intrinsic geometric flows, Ricci flow \cite{ham1} and Yamabe flow \cite{ham2}
which are evolution equations for Riemannian metrics. In a $2$-dimensional manifold, Ricci and Yamabe flow are equivalent, but not for higher dimensions.
Recently, in 2019, M. Crasmareanu and S. G\"{u}ler considered a scalar combination of Ricci and Yamabe flow called Ricci-Yamabe flow of type $(\alpha,\beta)$ \cite{cg}, for $\alpha$ and $\beta$ real numbers.
Due to the multiple choices of these scalars, the Ricci-Yamabe flow proves to be useful in different geometrical and physical theories. Let us remark that an interpolation flow between Ricci and Yamabe flows has been proposed by J.-P. Bourguignon in \cite{cal} under the name Ricci-Bourguignon flow, but which depends on a single scalar, studied further by E. Calvino-Louzao, E. Garcia-Rio, P. Gilkey, J.H. Park and R. V\'{a}zquez-Lorenzo in \cite{calv}, and by G. Catino, L. Cremaschi, Z. Djadli, C. Mantegazza and L. Mazzieri in \cite{r}. Also, unormalized version of Ricci-Bourguignon-Yamabe flow has been considered in \cite{mes} from the point of view of spectral geometry. Generalized fixed points of the above mentioned flows, the corresponding solitons, Ricci, Yamabe or Ricci-Bourguignon solitons with different kind of potential vector fields are investigated a lot in the Riemannian and pseudo-Riemannian setting, on manifolds endowed with various geometrical structures, and obstructions on the Ricci and scalar curvatures are consequently determined.

\smallskip

On the other hand, the nonlinear Schr\"{o}dinger-Ricci equation describes the evolution of the envelope of modulated wave groups, underlying the wave mechanics.
In Riemannian geometry, the Schr\"{o}dinger-Ricci equation involving the Ricci curvature tensor, the Laplace-Hodge operator and the Lie derivative of the metric, is subtle connected to the soliton equation and different types of harmonicity.

\smallskip

Considering these, we propose in the present paper to bring into light some aspects relating a more general notion of soliton, which we introduce here, harmonic forms and the solutions of the Schr\"{o}dinger-Ricci equation. Precisely, we provide necessary and sufficient conditions for the dual $1$-form of the potential vector field of a generalized soliton to be a solution of the Schr\"{o}dinger-Ricci equation, a harmonic form or a Schr\"{o}dinger-Ricci harmonic form.
Note that some partial results for the case of $\eta$-Ricci solitons were obtained in \cite{blagha}.
Concerning the particular case of Ricci solitons, we prove the followings: \textit{the dual $1$-form of the potential vector field is always a Schr\"{o}dinger-Ricci harmonic form} (Theorem \ref{te44}), \textit{on a complete Riemannian manifold, the dual $1$-form of the potential vector field is a solution of the Schr\"{o}dinger-Ricci equation if and only if the scalar curvature is constant} (Theorem \ref{t166}), \textit{if the potential vector field is of constant length and its dual $1$-form is a harmonic form, then the Ricci soliton is steady} (Theorem \ref{te1}).
For generalized Yamabe solitons, we show that \textit{the dual $1$-form of the potential vector field is a solution of the Schr\"{o}dinger-Ricci equation and a Schr\"{o}dinger-Ricci harmonic form} (Theorem \ref{tt} and Theorem \ref{tp}).

\section{Preliminaries}

Throughout the paper all the manifolds are assumed to be smooth. Let $(M,g)$ be an $n$-dimensional Riemannian manifold and denote by $\Ric$ and $\scal$ its Ricci curvature tensor and scalar curvature, respectively. We denote the set of all smooth vector fields of $M$ by $\mathfrak{X}(M)$.
Let $\eta$ be a $1$-form and $\xi$ a vector field on $M$, and denote by ${\mathcal L} _{\xi}$ the Lie derivative operator in the direction of $\xi$.
In \cite{bcd} we considered the following notion.

\begin{definition}
We say that $(g,\xi,\eta,\beta,\gamma,\delta)$ defines a \textit{generalized soliton} on a Riemannian manifold $(M,g)$ if
\begin{equation}\label{generalsol}
\frac{1}{2}{\mathcal L} _{\xi}g+\beta\cdot\Ric=\gamma\cdot g+\delta\cdot \eta\otimes \eta,
\end{equation}
with $\beta,\gamma,\delta$ smooth functions on $M$.
\end{definition}

In all the rest of this paper, we shall assume that $\eta:=\xi^{\flat}$ is the dual $1$-form of the potential vector field $\xi$ of the generalized soliton and simply denote it by $(M^n,g,\xi,\beta,\gamma,\delta)$.
Moreover, if $\beta=1$, we talk about an almost $\xi^{\flat}$-Ricci soliton and if $\beta=0$, we talk about an almost $\xi^{\flat}$-Yamabe soliton.
In both of the cases, if $\gamma$ and $\delta$ are constant, we just drop the adjective \textit{almost} and call the soliton $\xi^{\flat}$-Ricci and $\xi^{\flat}$-Yamabe soliton, respectively. Also, if $\beta=\delta=0$, we shall say that $(g,\xi,\gamma)$ defines an almost generalized Yamabe soliton or a generalized Yamabe soliton on $M$ provided $\gamma$ is a constant.

\bigskip

If the vector field $\xi$ is of gradient type, i.e., $\xi:=\grad(f)$, for a smooth function $f$ on $M$, and if we denote by $\Hess(f)$ the Hessian of $f$,
by $\nabla$ the Levi-Civita connection of $g$ and by $Q$ the Ricci operator, i.e., $$g(QX,Y):=\Ric(X,Y)$$ for $X$, $Y\in \mathfrak{X}(M)$,
then ${\mathcal L} _{\xi}g=2\Hess(f)$ and thus (\ref{generalsol}) takes the form
\begin{equation}\label{generalgr}
\Hess(f)+\beta\cdot\Ric=\gamma\cdot g+\delta\cdot \eta\otimes \eta,
\end{equation}
with $\beta,\gamma,\delta$ smooth functions on $M$, which is equivalent to
\begin{equation}\label{h}
\nabla \grad(f)+\beta\cdot Q=\gamma\cdot I+\delta\cdot \eta\otimes \xi,
\end{equation}
where $I$ is the identity endomorphism on $\mathfrak{X}(M)$, and we say that
$(g,\grad(f),\eta,\beta,\gamma,\delta)$ defines a \textit{gradient generalized soliton} on $(M,g)$.

\bigskip

A generalized soliton with $\delta=0$ will be called \textit{steady} if $\gamma=0$, \textit{shrinking} if $\gamma>0$, \textit{expanding} if $\gamma<0$, or \textit{undefined}, otherwise.

\bigskip

By taking the trace and by taking the divergence of (\ref{generalsol}), we consequently obtain
\begin{equation}\label{tr}
\div(\xi)+\beta \scal=n\gamma+\delta |\xi|^2,
\end{equation}
\begin{equation}\label{di}
\frac{1}{2}\div({\mathcal L} _{\xi}g)+\frac{\beta}{2}d(\scal)+i_{Q(\grad(\beta))}g=d\gamma+\eta(\grad(\delta))\eta+\delta\div(\xi)\eta+\delta i_{\nabla_{\xi}\xi}g.
\end{equation}

Also from (\ref{generalsol}) we find
\begin{equation}\label{po}
\beta\, i_{Q(\grad(\beta))}g=\gamma\, i_{\grad(\beta)}g+\delta \eta(\grad(\beta))i_{\xi}g-\frac{1}{2}\left(i_{\nabla_{\grad(\beta)}\xi}g+d\beta\circ \nabla \xi\right)
\end{equation}
and
\begin{equation}\label{pin}
\beta i_{Q\xi}g=(\gamma +\delta |\xi|^2)\eta-\frac{1}{2}\left(i_{\nabla_{\xi}\xi}g+\eta\circ \nabla \xi\right).
\end{equation}

\section{Harmonic aspects in a generalized soliton}\label{Section4}

Consider $(M,g)$ an $n$-dimensional Riemannian manifold, $n\geq 3$, and denote by
$$\flat:TM\rightarrow T^*M, \ \ \flat(X):=i_Xg, \ \ \sharp:T^*M\rightarrow TM, \ \ \sharp:=\flat^{-1}$$
the musical isomorphisms. Further, we shall use the notations $X^{\flat}=:\flat(X)$ and $\theta^{\sharp}=:\sharp(\theta)$.

Let $\mathcal{T}^0_{2,s}(M)$ be the set of symmetric $(0,2)$-tensor fields on $M$ and for $Z\in \mathcal{T}^0_{2,s}(M)$, denote by $Z^{\sharp}:TM\rightarrow TM$ and by $Z_{\sharp}:T^*M\rightarrow T^*M$ the maps defined as follows
$$g(Z^{\sharp}(X),Y):=Z(X,Y), \ \ Z_{\sharp}(\theta)(X):=Z(\sharp(\theta),X).$$
We also denote by $Z^{\sharp}$ the map $Z^{\sharp}:T^*M\times T^*M\rightarrow C^{\infty}(M)$
$$Z^{\sharp}(\theta_1,\theta_2):=Z(\sharp(\theta_1),\sharp(\theta_2))$$
and identify $Z_{\sharp}$ with the map, also denoted by $Z_{\sharp}:T^*M\times TM\rightarrow C^{\infty}(M)$
$$Z_{\sharp}(\theta,X):=Z_{\sharp}(\theta)(X).$$

It is known that \cite{cho}
\begin{equation}\label{e61}
\div({\mathcal L}_Xg)=(\Delta+\Ric_{\sharp})(X^{\flat})+d(\div(X)),
\end{equation}
where
$\Delta$ is the Laplace-Hodge operator on differential forms with respect to the metric $g$.
By a direct computation we deduce
$$\Ric_{\sharp}(\theta)=i_{Q(\theta^{\sharp})}g,$$ for any $1$-form $\theta$.

We are interested to find necessary and sufficient conditions for the dual $1$-form $\xi^{\flat}$ of the potential vector field $\xi$ of a generalized soliton
to be a solution of the Schr\"{o}dinger-Ricci equation, a harmonic form or a Schr\"{o}dinger-Ricci harmonic form.

\subsection{Schr\"{o}dinger-Ricci solutions}\label{Subection4.1}

A $1$-form $\theta$ on a Riemannian manifold $(M,g)$ is a \textit{solution of the Schr\"{o}dinger-Ricci equation} if
$$
(\Delta+\Ric_{\sharp})(\theta)+d(\div(\theta^{\sharp}))=0.
$$

\begin{lemma}
A $1$-form $\theta$ on a Riemannian manifold $(M,g)$ is a solution of the Schr\"{o}dinger-Ricci equation if and only if the dual vector field $\theta^{\sharp}$ of $\theta$ satisfies
\begin{equation}\label{l1}
\div({\mathcal L}_{\theta^{\sharp}}g)=0.
\end{equation}
\end{lemma}

We deduce the following proposition from (\ref{tr}), (\ref{di}), (\ref{po}) and (\ref{l1}).

\begin{proposition}\label{t1}
Let $(M^n, g,\xi,\beta,\gamma,\delta)$ be a generalized soliton such that $\beta$ is nowhere zero on $M$.
Then the dual $1$-form $\xi^{\flat}$ of the potential vector field $\xi$ is a solution of the Schr\"{o}dinger-Ricci equation if and only if it satisfies
$$
-\frac{\beta}{2}(d(\scal)+2\delta \scal\cdot\xi^{\flat})=\frac{\gamma}{\beta} d\beta-d\gamma-\frac{1}{2\beta}(\nabla_{\grad(\beta)}\xi^{\flat}+d\beta\circ \nabla \xi)
$$
\begin{equation}\label{kk}
+\left(\frac{\delta}{\beta}\xi^{\flat}(\grad(\beta))-\xi^{\flat}(\grad(\delta))\right)\xi^{\flat}-
\delta(n\gamma+\delta|\xi|^2)\xi^{\flat}-\delta \nabla_{\xi}\xi^{\flat}.
\end{equation}
\end{proposition}
\begin{proof}
From (\ref{l1}) we know that $\xi^{\flat}$ is a solution of the Schr\"{o}dinger-Ricci equation if and only if $\div({\mathcal L}_{\xi}g)=0$, which by means of (\ref{tr}), (\ref{di}) and (\ref{po}), is equivalent to
\begin{equation}\begin{aligned} &\nonumber
0=-\frac{\beta}{2}d(\scal)-i_{Q(\grad(\beta))}g
+d\gamma+\xi^{\flat}(\grad(\delta))\xi^{\flat}+\delta\div(\xi)\xi^{\flat}+\delta i_{\nabla_{\xi}\xi}g
\\& \hskip.3in =-\frac{\beta}{2}d(\scal)-\frac{\gamma}{\beta} i_{\grad(\beta)}g-\frac{\delta}{\beta} \xi^{\flat}(\grad(\beta))i_{\xi}g+\frac{1}{2\beta}\left(i_{\nabla_{\grad(\beta)}\xi}g+d\beta\circ \nabla \xi\right)+
d\gamma
\\& \hskip.3in +\xi^{\flat}(\grad(\delta))\xi^{\flat}+\delta(n\gamma+\delta |\xi|^2-\beta \scal)\xi^{\flat}+\delta i_{\nabla_{\xi}\xi}g
\\& \hskip.3in =-\frac{\beta}{2}d(\scal)-\frac{\gamma}{\beta} d\beta-\frac{\delta}{\beta} \xi^{\flat}(\grad(\beta))\xi^{\flat}+\frac{1}{2\beta}\left(\nabla_{\grad(\beta)}\xi^{\flat}+d\beta\circ \nabla \xi\right)+
d\gamma
\\& \hskip.3in +\xi^{\flat}(\grad(\delta))\xi^{\flat}+\delta(n\gamma+\delta |\xi|^2-\beta \scal)\xi^{\flat}+\delta \nabla_{\xi}\xi^{\flat}.
\end{aligned}
\end{equation}
\end{proof}

Proposition \ref{t1} implies the following.

\begin{theorem}
Let $(M,g)$ be a Riemannian manifold of constant scalar curvature and $\xi$ a constant length vector field such that its dual $1$-form $\xi^{\flat}$ is a solution of the Schr\"{o}dinger-Ricci equation. If $(g,\xi,\gamma,\delta)$ defines a $\xi^{\flat}$-Ricci soliton on $M$, then either the soliton is a Ricci soliton or $\xi$ is a divergence-free vector field.
\end{theorem}
\begin{proof}
Taking into account that $\beta=1$, $\gamma$, $\delta$ and $\scal$ are constant, we obtain
$$\delta\left(\div(\xi)\xi^{\flat}+\nabla_{\xi}\xi^{\flat}\right)=0$$
from (\ref{kk}), which either implies $\delta=0$ or $\div(\xi)\xi^{\flat}=-\nabla_{\xi}\xi^{\flat}$. Computing the second relation in $\xi$, we find
$$\div(\xi)=-\frac{\xi(|\xi|^2)}{2|\xi|^2}=0.$$
\end{proof}

\begin{theorem}\label{t166}
Let $(M,g)$ be a complete Riemannian manifold.
If $(M,g,\xi,\gamma)$ is a Ricci soliton, then the dual $1$-form $\xi^{\flat}$ of the potential vector field $\xi$ is a solution of the Schr\"{o}dinger-Ricci equation if and only if the scalar curvature of $M$ is constant.
\end{theorem}
\begin{proof}
For $\beta=1$, $\delta=0$ and $\gamma$ constant, we deduce that (\ref{kk}) holds if and only if
$d(\scal)=0$, hence the conclusion.
\end{proof}

Similarly, we deduce the following proposition from (\ref{tr}), (\ref{di}) for $\beta=0$, and (\ref{l1}).

\begin{proposition}\label{p11}
Let $(M^n, g,\xi,\gamma,\delta)$ be an almost $\xi^{\flat}$-Yamabe soliton.
Then the dual $1$-form $\xi^{\flat}$ of the potential vector field $\xi$ is a solution of the Schr\"{o}dinger-Ricci equation if and only if it satisfies
\begin{equation}\label{p1}
d\gamma+\xi^{\flat}(\grad(\delta))\xi^{\flat}+
\delta(n\gamma+\delta|\xi|^2)\xi^{\flat}+\delta \nabla_{\xi}\xi^{\flat}=0.
\end{equation}
\end{proposition}
\begin{proof}
From (\ref{l1}) we know that $\xi^{\flat}$ is a solution of the Schr\"{o}dinger-Ricci equation if and only if $\div({\mathcal L}_{\xi}g)=0$, which by means of (\ref{tr}) and (\ref{di}) for $\beta=0$, is equivalent to
\begin{equation}\begin{aligned} &\nonumber
0=d\gamma+\xi^{\flat}(\grad(\delta))\xi^{\flat}+\delta\div(\xi)\xi^{\flat}+\delta i_{\nabla_{\xi}\xi}g
\\& \hskip.3in =d\gamma+\xi^{\flat}(\grad(\delta))\xi^{\flat}+\delta(n\gamma+\delta|\xi|^2)\xi^{\flat}+\delta \nabla_{\xi}\xi^{\flat}.
\end{aligned}
\end{equation}
\end{proof}

We immediately obtain the followings from Proposition \ref{p11}.

\begin{theorem}
Let $(M,g,\xi,\gamma,\delta)$ be a $\xi^{\flat}$-Yamabe soliton whose potential vector field $\xi$ is of constant length. If the dual $1$-form $\xi^{\flat}$ of $\xi$ is a solution of the Schr\"{o}dinger-Ricci equation, then either the soliton is a generalized Yamabe soliton or $\xi$ is a divergence-free vector field.
\end{theorem}
\begin{proof}
Taking into account that $\gamma$ and $\delta$ are constant, we obtain
$$\delta\left(\div(\xi)\xi^{\flat}+\nabla_{\xi}\xi^{\flat}\right)=0$$
from (\ref{p1}), which either implies $\delta=0$ or $\div(\xi)\xi^{\flat}=-\nabla_{\xi}\xi^{\flat}$. Computing the second relation in $\xi$, we find
$$\div(\xi)=-\frac{\xi(|\xi|^2)}{2|\xi|^2}=0.$$
\end{proof}

\begin{theorem}\label{c:4.7}
Let $(M,g)$ be a complete Riemannian manifold.
If $(M,g,\xi,\gamma)$ is an almost generalized Yamabe soliton, then the dual $1$-form $\xi^{\flat}$ of the potential vector field $\xi$ is a solution of the Schr\"{o}dinger-Ricci equation if and only if the soliton is a generalized Yamabe soliton.
\end{theorem}
\begin{proof}
For $\delta=0$, we deduce that (\ref{p1}) holds if and only if $d\gamma=0$, hence the conclusion.
\end{proof}

\begin{theorem}\label{tt}
For every generalized Yamabe soliton $(M,g,\xi,\gamma)$, the dual $1$-form $\xi^{\flat}$ of the potential vector field $\xi$ is a solution of the Schr\"{o}dinger-Ricci equation.
\end{theorem}

\subsection{Schr\"{o}dinger-Ricci harmonic forms}\label{Subection4.2}

A $1$-form $\theta$ on a Riemannian manifold $(M,g)$ is called \textit{Schr\"{o}dinger-Ricci harmonic} if it  satisfies
$$(\Delta+\Ric_{\sharp})(\theta)=0.$$

\begin{lemma}
A $1$-form $\theta$ on a Riemannian manifold $(M,g)$ is a Schr\"{o}dinger-Ricci harmonic form if and only if
the dual vector field $\theta^{\sharp}$ of $\theta$ satisfies \begin{equation}\label{l2}
\div({\mathcal L}_{\theta^{\sharp}}g)=d(\div(\theta^{\sharp})).
\end{equation}
\end{lemma}

Since
$$d(\div(\theta^{\sharp}))=i_{\grad(\div(\theta^{\sharp}))}g, \ \ \theta\circ \nabla \theta^{\sharp}=\frac{1}{2}d(|\theta^{\sharp}|^2), \ \ (\nabla_X\theta)^{\sharp}=\nabla_X\theta^{\sharp},$$
we deduce the following proposition from (\ref{tr}), (\ref{di}), (\ref{po}) and (\ref{l2}).

\begin{proposition}\label{te}
Let $(M^n,g,\xi,\beta,\gamma,\delta)$ be a generalized soliton such that $\beta$ is nowhere zero on $M$.
Then the dual $1$-form $\xi^{\flat}$ of the potential vector field $\xi$ is a Schr\"{o}dinger-Ricci harmonic form if and only if it satisfies
$$
\frac{\scal}{2}(d\beta-2\beta\delta \xi^{\flat})=\frac{\gamma}{\beta}d\beta+\frac{n-2}{2}d\gamma+\frac{|\xi|^2}{2}d\delta-\frac{1}{2\beta}(\nabla_{\grad(\beta)}\xi^{\flat}+d\beta\circ \nabla \xi)$$
\begin{equation}\label{kk2}
+\left(\frac{\delta}{\beta}\xi^{\flat}(\grad(\beta))-\xi^{\flat}(\grad(\delta))\right)\xi^{\flat}-
\delta(n\gamma+\delta|\xi|^2)\xi^{\flat}-\delta(\nabla_{\xi}\xi^{\flat}-\xi^{\flat}\circ \nabla \xi).
\end{equation}
\end{proposition}
\begin{proof}
From (\ref{l2}) we know that $\xi^{\flat}$ is a Schr\"{o}dinger-Ricci harmonic form if and only if $\div({\mathcal L}_{\xi}g)=d(\div(\xi))$, which by means of
(\ref{tr}), (\ref{di}) and (\ref{po}), is equivalent to
\begin{equation}\begin{aligned} &\nonumber
0=-\beta d(\scal)-2i_{Q(\grad(\beta))}g
+2d\gamma+2\xi^{\flat}(\grad(\delta))\xi^{\flat}+2\delta\div(\xi)\xi^{\flat}+2\delta i_{\nabla_{\xi}\xi}g
\\& \hskip.3in -nd\gamma-\delta d(|\xi|^2)-|\xi|^2d\delta +\beta d(\scal)+\scal \cdot d\beta
\\& \hskip.3in =-\frac{2\gamma}{\beta} i_{\grad(\beta)}g-\frac{2\delta}{\beta} \xi^{\flat}(\grad(\beta))i_{\xi}g+
\frac{1}{\beta}\left(i_{\nabla_{\grad(\beta)}\xi}g+d\beta\circ \nabla \xi\right)
\\& \hskip.3in +2d\gamma+2\xi^{\flat}(\grad(\delta))\xi^{\flat}+2\delta(n\gamma+\delta |\xi|^2-\beta \scal)\xi^{\flat}+2\delta i_{\nabla_{\xi}\xi}g
\\& \hskip.3in -nd\gamma-\delta d(|\xi|^2)-|\xi|^2d\delta +\scal \cdot d\beta
\\& \hskip.3in =-\frac{2\gamma}{\beta} d\beta-\frac{2\delta}{\beta} \xi^{\flat}(\grad(\beta))\xi^{\flat}+\frac{1}{\beta}\left(\nabla_{\grad(\beta)}\xi^{\flat}+d\beta\circ \nabla \xi\right)-(n-2)d\gamma
\\& \hskip.3in +2\xi^{\flat}(\grad(\delta))\xi^{\flat}+2\delta(n\gamma+\delta |\xi|^2)\xi^{\flat}+2\delta (\nabla_{\xi}\xi^{\flat}-\xi^{\flat}\circ \nabla \xi)-|\xi|^2d\delta +\scal (d\beta-2\beta\delta \xi^{\flat}).
\end{aligned}\end{equation}
\end{proof}

Proposition \ref{te} implies the following.

\begin{theorem}
Let $(M,g)$ be a Riemannian manifold and $\xi$ a vector field such that its dual $1$-form $\xi^{\flat}$ is a Schr\"{o}dinger-Ricci harmonic form. If $(g,\xi,\gamma,\delta)$ defines a $\xi^{\flat}$-Ricci soliton on $M$, then either the soliton is a Ricci soliton or $\xi$ is a divergence-free vector field.
\end{theorem}
\begin{proof}
Taking into account that $\beta=1$, $\gamma$ and $\delta$ are constant, we obtain
$$\delta\left(\div(\xi)\xi^{\flat}+\nabla_{\xi}\xi^{\flat}-\xi^{\flat}\circ \nabla \xi\right)=0$$
from (\ref{kk2}), which either implies $\delta=0$ or $\div(\xi)\xi^{\flat}=\xi^{\flat}\circ \nabla \xi-\nabla_{\xi}\xi^{\flat}$. Computing the second relation in $\xi$, we find
$$\div(\xi)=0.$$
\end{proof}

\begin{theorem}\label{te44} For every Ricci soliton $(M,g,\xi,\gamma)$, the dual $1$-form $\xi^{\flat}$ of the potential vector field $\xi$ is a Schr\"{o}dinger-Ricci harmonic form.
\end{theorem}
\begin{proof}
For $\beta=1$, $\delta=0$ and $\gamma$ constant, we deduce that (\ref{kk2}) always holds, hence the conclusion.
\end{proof}

Similarly, we deduce the following proposition from (\ref{tr}), (\ref{di}) for $\beta=0$, and (\ref{l2}).

\begin{proposition}\label{p22}
Let $(M^n, g,\xi,\gamma,\delta)$ be an almost $\xi^{\flat}$-Yamabe soliton.
Then the dual $1$-form $\xi^{\flat}$ of the potential vector field $\xi$ is a Schr\"{o}dinger-Ricci harmonic form if and only if it satisfies
\begin{equation}\label{p2}
\frac{n-2}{2}d\gamma+\frac{|\xi|^2}{2}d\delta-\xi^{\flat}(\grad(\delta))\xi^{\flat}-
\delta(n\gamma+\delta|\xi|^2)\xi^{\flat}-\delta(\nabla_{\xi}\xi^{\flat}-\xi^{\flat}\circ \nabla \xi)=0.
\end{equation}
\end{proposition}
\begin{proof}
From (\ref{l2}) we know that $\xi^{\flat}$ is a Schr\"{o}dinger-Ricci harmonic form if and only if $\div({\mathcal L}_{\xi}g)=d(\div(\xi))$, which by means of
(\ref{tr}) and (\ref{di}) for $\beta=0$, is equivalent to
\begin{equation}\begin{aligned} &\nonumber
0=2d\gamma+2\xi^{\flat}(\grad(\delta))\xi^{\flat}+2\delta\div(\xi)\xi^{\flat}+2\delta i_{\nabla_{\xi}\xi}g
-nd\gamma-\delta d(|\xi|^2)-|\xi|^2d\delta
\\& \hskip.3in =2d\gamma+2\xi^{\flat}(\grad(\delta))\xi^{\flat}+2\delta(n\gamma+\delta|\xi|^2)\xi^{\flat}+2\delta \nabla_{\xi}\xi^{\flat}
-nd\gamma-2\delta \xi^{\flat}\circ \nabla \xi-|\xi|^2d\delta.
\end{aligned}\end{equation}
\end{proof}

We immediately obtain the followings from Proposition \ref{p22}.

\begin{theorem}
If $(M,g,\xi,\gamma,\delta)$ is a $\xi^{\flat}$-Yamabe soliton such that the dual $1$-form $\xi^{\flat}$ of the potential vector field $\xi$ is a Schr\"{o}dinger-Ricci harmonic form, then either the soliton is a generalized Yamabe soliton or $\xi$ is a divergence-free vector field.
\end{theorem}
\begin{proof}
Taking into account that $\gamma$ and $\delta$ are constant, we obtain
$$\delta\left(\div(\xi)\xi^{\flat}+\nabla_{\xi}\xi^{\flat}-\xi^{\flat}\circ \nabla \xi\right)=0$$
from (\ref{p2}), which either implies $\delta=0$ or $\div(\xi)\xi^{\flat}=\xi^{\flat}\circ \nabla \xi-\nabla_{\xi}\xi^{\flat}$. Computing the second relation in $\xi$, we find
$$\div(\xi)=0.$$
\end{proof}

\begin{theorem}\label{c:4.14}
Let $(M^n,g)$ be a complete Riemannian manifold, $n>2$.
If $(M^n,g,\xi,\gamma)$ is an almost generalized Yamabe soliton, then the dual $1$-form $\xi^{\flat}$ of the potential vector field $\xi$ is a Schr\"{o}dinger-Ricci harmonic form if and only if the soliton is a generalized Yamabe soliton.
\end{theorem}
\begin{proof}
For $\delta=0$, we deduce that (\ref{p2}) holds if and only if $(n-2)d\gamma=0$, hence the conclusion.
\end{proof}

\begin{theorem}\label{tp}
For every generalized Yamabe soliton $(M,g,\xi,\gamma)$, the dual $1$-form $\xi^{\flat}$ of the potential vector field $\xi$ is a Schr\"{o}dinger-Ricci harmonic form.
\end{theorem}

\subsection{Harmonic forms}\label{Subection4.3}

A $1$-form $\theta$ on a Riemannian manifold $(M,g)$ is called \textit{harmonic} if it satisfies $$\Delta(\theta)=0.$$

\smallskip

We can notice that if $\theta^{\sharp}\in\ker Q$, then $\theta$ is Schr\"{o}dinger-Ricci harmonic if and only if it is harmonic.

\begin{lemma}
A $1$-form $\theta$ on a Riemannian manifold $(M,g)$ is a harmonic form if and only if the dual vector field $\theta^{\sharp}$ of $\theta$ satisfies
\begin{equation}\label{l3}
\div({\mathcal L}_{\theta^{\sharp}}g)=\Ric_{\sharp}(\theta)+d(\div(\theta^{\sharp})).
\end{equation}
\end{lemma}

Since
$$\Ric_{\sharp}(\theta)=i_{Q(\theta^{\sharp})}g, \ \ d(\div(\theta^{\sharp}))=i_{\grad(\div(\theta^{\sharp}))}g, \ \ \theta\circ \nabla \theta^{\sharp}=\frac{1}{2}d(|\theta^{\sharp}|^2), \ \ (\nabla_X\theta)^{\sharp}=\nabla_X\theta^{\sharp},$$
we deduce the following proposition from (\ref{tr}), (\ref{di}), (\ref{po}), (\ref{pin}) and (\ref{l3}).

\begin{proposition}\label{teor} Let $(M^n,g,\xi,\beta,\gamma,\delta)$ be a generalized soliton such that $\beta$ is nowhere zero on $M$. Then the dual $1$-form $\xi^{\flat}$ of the potential vector field $\xi$ is a harmonic form if and only if it satisfies
$$
\frac{\scal}{2}(d\beta-2\beta\delta \xi^{\flat})=\frac{\gamma}{\beta} d\beta+\frac{n-2}{2}d\gamma+\frac{|\xi|^2}{2}d\delta-\frac{1}{2\beta}(\nabla_{\grad(\beta)}\xi^{\flat}+d\beta\circ \nabla \xi)$$$$+\left(\frac{\delta}{\beta}\xi^{\flat}(\grad(\beta))-\xi^{\flat}(\grad(\delta))\right)\xi^{\flat}-
\frac{(2n\beta\delta-1)\gamma+\delta(2\beta\delta-1)|\xi|^2}{2\beta}\xi^{\flat}$$
\begin{equation}\label{kk3}
-\left(\delta+\frac{1}{4\beta}\right)\nabla_{\xi}\xi^{\flat}+
\left(\delta-\frac{1}{4\beta}\right)\xi^{\flat}\circ \nabla \xi.
\end{equation}
\end{proposition}
\begin{proof}
From (\ref{l3}) we know that $\xi^{\flat}$ is a harmonic form if and only if $\div({\mathcal L}_{\xi}g)=\Ric_{\sharp}(\xi^{\flat})+d(\div(\xi))$, which by means of
(\ref{tr}), (\ref{di}), (\ref{po}) and (\ref{pin}), is equivalent to
\begin{equation}\begin{aligned} &\nonumber
0=-\beta d(\scal)-2i_{Q(\grad(\beta))}g
+2d\gamma+2\xi^{\flat}(\grad(\delta))\xi^{\flat}+2\delta\div(\xi)\xi^{\flat}+2\delta i_{\nabla_{\xi}\xi}g
\\& \hskip.3in -i_{Q\xi}g-nd\gamma-\delta d(|\xi|^2)-|\xi|^2d\delta +\beta d(\scal)+\scal \cdot d\beta
\\& \hskip.3in =-\frac{2\gamma}{\beta} i_{\grad(\beta)}g-\frac{2\delta}{\beta} \xi^{\flat}(\grad(\beta))i_{\xi}g+
\frac{1}{\beta}\left(i_{\nabla_{\grad(\beta)}\xi}g+d\beta\circ \nabla \xi\right)
\\& \hskip.3in +2d\gamma+2\xi^{\flat}(\grad(\delta))\xi^{\flat}+2\delta(n\gamma+\delta |\xi|^2-\beta \scal)\xi^{\flat}+2\delta i_{\nabla_{\xi}\xi}g
\\& \hskip.3in -\frac{1}{\beta}(\gamma+\delta|\xi|^2)\xi^{\flat}+\frac{1}{2\beta}(i_{\nabla_{\xi}\xi}g+\xi^{\flat}\circ \nabla \xi)-nd\gamma-\delta d(|\xi|^2)-|\xi|^2d\delta +\scal \cdot d\beta
\\& \hskip.3in =-\frac{2\gamma}{\beta} d\beta-\frac{2\delta}{\beta} \xi^{\flat}(\grad(\beta))\xi^{\flat}+\frac{1}{\beta}\left(\nabla_{\grad(\beta)}\xi^{\flat}+d\beta\circ \nabla \xi\right)-(n-2)d\gamma
\\& \hskip.3in +2\xi^{\flat}(\grad(\delta))\xi^{\flat}+\Big(\Big(2n\delta-\frac{1}{\beta}\Big)\gamma+\Big(2\delta-\frac{1}{\beta}\Big) \delta|\xi|^2\Big)\xi^{\flat}
\\& \hskip.3in +\left(2\delta+\frac{1}{2\beta}\right) \nabla_{\xi}\xi^{\flat}-\left(2\delta-\frac{1}{2\beta}\right) \xi^{\flat}\circ \nabla \xi-|\xi|^2d\delta +\scal (d\beta-2\beta\delta \xi^{\flat}).
\end{aligned}\end{equation}
\end{proof}

Proposition \ref{teor} implies the following.

\begin{theorem} \label{te1}
Let $(M,g,\xi,\gamma)$ be a Ricci soliton whose potential vector field $\xi$ is of constant length. If the dual $1$-form $\xi^{\flat}$ of $\xi$ is a harmonic form, then the Ricci soliton is steady.
\end{theorem}
\begin{proof}
For $\beta=1$, $\delta=0$ and $\gamma$ constant, we obtain
$$\gamma  \xi^{\flat}=\frac{1}{2}\left(\nabla_{\xi}\xi^{\flat}+\xi^{\flat}\circ \nabla \xi\right).$$
from (\ref{kk3}). Computing this relation in $\xi$, we find
$$\gamma=\frac{\xi(|\xi|^2)}{2|\xi|^2}=0.$$
\end{proof}

Similarly, we deduce the following proposition from (\ref{tr}), (\ref{di}) for $\beta=0$, and (\ref{l3}).

\begin{proposition}\label{p33}
Let $(M^n, g,\xi,\gamma,\delta)$ be an almost $\xi^{\flat}$-Yamabe soliton.
Then the dual $1$-form $\xi^{\flat}$ of the potential vector field $\xi$ is a harmonic form if and only if it satisfies
$$Q\xi=-(n-2)\grad(\gamma)-|\xi|^2\grad(\delta)+2\xi^{\flat}(\grad(\delta))\xi
$$
\begin{equation}\label{p4}
+2\delta\left(\nabla_{\xi}\xi-\frac{1}{2}\grad(|\xi|^2)+(n\gamma+\delta|\xi|^2)\xi\right).
\end{equation}
\end{proposition}
\begin{proof}
From (\ref{l3}) we know that $\xi^{\flat}$ is a harmonic form if and only if $\div({\mathcal L}_{\xi}g)=\Ric_{\sharp}(\xi^{\flat})+d(\div(\xi))$, which by means of
(\ref{tr}), (\ref{di}) for $\beta=0$, and (\ref{pin}), is equivalent to
\begin{equation}\begin{aligned} &\nonumber
0=2d\gamma+2\xi^{\flat}(\grad(\delta))\xi^{\flat}+2\delta\div(\xi)\xi^{\flat}+2\delta i_{\nabla_{\xi}\xi}g-i_{Q\xi}g-nd\gamma-\delta d(|\xi|^2)-|\xi|^2d\delta
\\& \hskip.3in =2i_{\grad(\gamma)}g+2\xi^{\flat}(\grad(\delta))i_{\xi}g+2\delta(n\gamma+\delta|\xi|^2)i_{\xi}g+2\delta i_{\nabla_{\xi}\xi}g-i_{Q\xi}g
\\& \hskip.3in -ni_{\grad(\gamma)}g-\delta i_{\grad(|\xi|^2)}g-|\xi|^2i_{\grad(\delta)}g.
\end{aligned}\end{equation}
\end{proof}

We immediately obtain the following from Proposition \ref{p33}.

\begin{corollary}
If $(M,g,\xi,\gamma)$ is an almost generalized Yamabe soliton, then the dual $1$-form $\xi^{\flat}$ of the potential vector field $\xi$ is a harmonic form if and only if
$$Q\xi=-(n-2)\grad(\gamma).$$

In particular, if $(M,g,\xi,\gamma)$ is a generalized Yamabe soliton, then $\xi^{\flat}$ is a harmonic form if and only if
$\xi \in \ker Q$.
\end{corollary}
\begin{proof}
For $\delta=0$, we deduce that (\ref{p4}) holds if and only if $Q\xi=-(n-2)\grad(\gamma)$. In particular, if $\gamma$ is constant, then the second assertion follows, too.
\end{proof}

Using the fact that $(\nabla_X\theta)^{\sharp}=\nabla_X\theta^{\sharp}$, for any vector field $X$ and any $1$-form $\theta$, we conclude

\begin{proposition} Let $(M,g,\xi,\beta,\gamma,\delta)$ be a generalized soliton such that $\beta$ is nowhere zero on $M$.
Then a harmonic $1$-form $\theta$ on $(M,g)$ is Schr\"{o}dinger-Ricci harmonic if and only if it satisfies
$$\gamma \theta+\delta \theta(\xi)\xi^{\flat}=\frac{1}{2}(\nabla_{\theta^{\sharp}}\xi^{\flat}+\theta\circ \nabla \xi).$$
\end{proposition}
\begin{proof}
From (\ref{l2}) and (\ref{l3}) we find that $\theta$ is Schr\"{o}dinger-Ricci harmonic if and only if $\Ric_{\sharp}(\theta)=0$,
which from the generalized soliton equation (\ref{generalsol}) is equivalent to
$$0=\beta i_{Q(\theta^{\sharp})}g=\gamma i_{\theta^{\sharp}}g+\delta \xi^{\flat}(\theta^{\sharp})\xi^{\flat}-\frac{1}{2}(i_{\nabla_{\theta^{\sharp}}\xi}g+\theta\circ \nabla \xi),$$
hence the conclusion.
\end{proof}

The relations between the above notions can be summarized as the following.

\begin{theorem}\label{t4}
Let $\theta$ be a $1$-form on a Riemannian manifold $(M,g)$. \\
{\rm (i)} If $\theta$ is a solution of the Schr\"{o}dinger-Ricci equation, then $\theta$ is

{\rm (a)} a Schr\"{o}dinger-Ricci harmonic form if and only if $d(\div(\theta^{\sharp}))=0$;

{\rm (b)} a harmonic form if and only if $\Ric_{\sharp}(\theta)+d(\div(\theta^{\sharp}))=0$.\\
{\rm (ii)} If $\theta$ is a Schr\"{o}dinger-Ricci harmonic form, then $\theta$ is

{\rm (a)} a solution of the Schr\"{o}dinger-Ricci equation if and only if $d(\div(\theta^{\sharp}))=0$;

{\rm (b)} a harmonic form if and only if $\Ric_{\sharp}(\theta)=0$.\\
{\rm (iii)} If $\theta$ is a harmonic form, then $\theta$ is

{\rm (a)} a solution of the Schr\"{o}dinger-Ricci equation if and only if $\Ric_{\sharp}(\theta)+d(\div(\theta^{\sharp}))=0$;

{\rm (b)} a Schr\"{o}dinger-Ricci harmonic form if and only if $\Ric_{\sharp}(\theta)=0$.
\end{theorem}

From (\ref{generalsol}), (\ref{tr}) and Theorem \ref{t4}, we find the relations between the cases when the dual $1$-form $\xi^{\flat}$ of the potential vector field $\xi$ is a solution of the Schr\"{o}dinger-Ricci equation, is a Schr\"{o}dinger-Ricci harmonic form or a harmonic form.
\begin{proposition}
Let $(M^n,g,\xi,\beta,\gamma,\delta)$ be a generalized soliton such that $\beta$ is nowhere zero on $M$.\\
{\rm (i)} If $\xi^{\flat}$ is a solution of the Schr\"{o}dinger-Ricci equation, then $\xi^{\flat}$ is

{\rm (a)} a Schr\"{o}dinger-Ricci harmonic form if and only if $$\beta\scal-n\gamma-\delta|\xi|^2 \ \ \text{is constant};$$

{\rm (b)} a harmonic form if and only if $$d(\beta\scal)=\frac{1}{\beta}(\gamma +\delta |\xi|^2)\xi^{\flat}+nd\gamma+|\xi|^2d\delta-\frac{1}{2\beta}\nabla_{\xi}\xi^{\flat}+\left(2\delta-\frac{1}{2\beta}\right)\xi^{\flat}\circ \nabla \xi.$$ \\
{\rm (ii)} If $\xi^{\flat}$ is a Schr\"{o}dinger-Ricci harmonic form, then $\xi^{\flat}$ is

{\rm (a)} a solution of the Schr\"{o}dinger-Ricci equation if and only if $$\beta\scal-n\gamma-\delta|\xi|^2 \ \ \text{is constant};$$

{\rm (b)} a harmonic form if and only if $$(\gamma +\delta |\xi|^2)\xi^{\flat}=\frac{1}{2}(\nabla_{\xi}\xi^{\flat}+\xi^{\flat}\circ \nabla \xi).$$\\
{\rm (iii)} If $\xi^{\flat}$ is a harmonic form, then $\xi^{\flat}$ is

{\rm (a)} a solution of the Schr\"{o}dinger-Ricci equation if and only if $$d(\beta\scal)=\frac{1}{\beta}(\gamma +\delta |\xi|^2)\xi^{\flat}+nd\gamma+|\xi|^2d\delta-\frac{1}{2\beta}\nabla_{\xi}\xi^{\flat}+\left(2\delta-\frac{1}{2\beta}\right)\xi^{\flat}\circ \nabla \xi;$$

{\rm (b)} a Schr\"{o}dinger-Ricci harmonic form if and only if $$(\gamma +\delta |\xi|^2)\xi^{\flat}=\frac{1}{2}(\nabla_{\xi}\xi^{\flat}+\xi^{\flat}\circ \nabla \xi).$$
\end{proposition}

\subsection{$1$-forms orthogonal to $\xi^{\flat}$}\label{Subsection3.4}

Two $1$-forms $\theta_1$ and $\theta_2$ are \textit{orthogonal} if $g(\theta_1^{\sharp},\theta_2^{\sharp})=0$, i.e., $\langle\theta_1,\theta_2\rangle=0$.

\smallskip

Since $\theta_1(\theta_2^{\sharp})=\theta_2(\theta_1^{\sharp})$, remark that $\theta_1$ and $\theta_2$ are orthogonal if and only if $$\theta_1^{\sharp}\in \ker \theta_2\ \ (\textit{or} \ \ \theta_2^{\sharp}\in \ker \theta_1).$$

Computing the generalized soliton equation (\ref{generalsol}) in $(\theta^{\sharp},\theta^{\sharp})$, we obtain
$$g(\nabla_{\theta^{\sharp}}\xi,\theta^{\sharp})+\beta g(Q(\theta^{\sharp}),\theta^{\sharp})-\gamma g(\theta^{\sharp},\theta^{\sharp})=\delta (\xi^{\flat}(\theta^{\sharp}))^2$$
and we can state

\begin{proposition}\label{j}
Let $(M,g,\xi,\beta,\gamma,\delta)$ be a generalized soliton such that $\delta$ is nowhere zero on $M$.
Then a $1$-form $\theta$ is orthogonal to $\xi^{\flat}$ if and only if the dual vector field $\theta^{\sharp}$ of $\theta$ satisfies
$$
\nabla_{\theta^{\sharp}}\xi+\beta Q(\theta^{\sharp})-\gamma \theta^{\sharp}\in \ker \theta.
$$
\end{proposition}

Computing now the same equation in $(\theta^{\sharp}, \cdot\,)$, we find
$$\frac{1}{2}(\nabla_{\theta^{\sharp}}\xi^{\flat}+\xi^{\flat}\circ \nabla \theta^{\sharp})=i_{\gamma \theta^{\sharp}-\beta Q(\theta^{\sharp})}g+\delta \xi^{\flat}(\theta^{\sharp})\xi^{\flat}$$
and we can state

\begin{proposition}
Let $(M,g,\xi,\beta,\gamma,\delta)$ be a generalized soliton such that $\delta$ is nowhere zero on $M$.
Then a $1$-form $\theta$ is orthogonal to $\xi^{\flat}$ if and only if the dual vector field $\theta^{\sharp}$ of $\theta$ satisfies
$$
\frac{1}{2}(\nabla_{\theta^{\sharp}}\xi^{\flat}+\xi^{\flat}\circ \nabla \theta^{\sharp})=i_{\gamma \theta^{\sharp}-\beta Q(\theta^{\sharp})}g.
$$
\end{proposition}

\section{The gradient case}\label{Section5}

Let $\xi=\grad(f)$, $\xi^{\flat}=df$ with $f$ a smooth function on $(M,g)$.
Taking into account that \cite{blag}
$$
\div({\mathcal L}_{\xi}g)=2 d(\div(\xi))+2i_{Q\xi}g
$$
and using (\ref{e61}):
$$2\div(\Hess(f))=\Delta(df)+\Ric_{\sharp}(df)+d(\Delta (f)),$$
we deduce
\begin{equation}\label{er}
\Delta(df)=d(\Delta(f))+\Ric_{\sharp}(df).
\end{equation}
Therefore, $df$ is

(a) a solution of the Schr\"{o}dinger-Ricci equation if and only if
$$\Delta(df)=0;$$

(b) a Schr\"{o}dinger-Ricci harmonic form if and only if
$$\Delta(df)=\frac{1}{2}d(\Delta(f));$$

(c) a harmonic form if and only if
$$df\circ Q=-d(\Delta(f)).$$

Also, an exact $1$-form $df$ is harmonic if and only if the function $f$ is harmonic and from (\ref{er}) we deduce that $df$ is harmonic if and only if
$$df\circ Q=0 \Longleftrightarrow df\in \ker (\Ric_{\sharp}).$$

\bigskip

Computing the gradient generalized soliton equation (\ref{h}) in $(\theta^{\sharp}, \cdot\,)$, we obtain
$$\nabla_{\theta^{\sharp}}\xi+\beta Q(\theta^{\sharp})=\gamma \theta^{\sharp}+\delta \xi^{\flat}(\theta^{\sharp})\xi$$
which, by taking the inner product with $\xi$, implies
$$\frac{1}{2}\theta^{\sharp}(|\xi|^2)+\beta\theta(Q\xi)=(\gamma+\delta|\xi|^2)\xi^{\flat}(\theta^{\sharp}).$$

Hence, we can state the followings.

\begin{proposition}\label{P:5.1}
Let $(M,g,\xi,\beta,\gamma,\delta)$ be a gradient generalized soliton such that $\delta$ is nowhere zero on $M$. Then a $1$-form $\theta$ is orthogonal to $\xi^{\flat}$ if and only if the dual vector field $\theta^{\sharp}$ of $\theta$ satisfies
$$
\nabla_{\theta^{\sharp}}\xi+\beta Q(\theta^{\sharp})-\gamma \theta^{\sharp}=0,
$$
hence we have
\begin{equation}\label{k}
\frac{1}{2}\theta^{\sharp}(|\xi|^2)=-\beta\theta(Q\xi).
\end{equation}
\end{proposition}

\begin{remark}
If $(g,\xi:=\grad(f),\beta,\gamma,\delta)$ defines a gradient generalized soliton on a compact and oriented manifold $M$, and the $1$-form $\theta$ is orthogonal to $\xi^{\flat}$, then
$$0=g(\nabla_{\xi}\xi,\theta^{\sharp})+\beta g(\xi,Q(\theta^{\sharp}))=
(\nabla_{\xi}\xi^{\flat})\theta^{\sharp}+\beta \xi^{\flat}(Q(\theta^{\sharp})),$$
hence
$$\theta^{\sharp}\in \ker(\nabla_{\xi}\xi^{\flat}+\beta (\xi^{\flat}\circ Q)).$$
\end{remark}

\begin{proposition}\label{P:5.11}
Let $(M,g,\xi,\gamma,\delta)$ be a gradient almost $\xi^{\flat}$-Yamabe soliton such that $\delta$ is nowhere zero on $M$. Then a $1$-form $\theta$ is orthogonal to $\xi^{\flat}$ if and only if the dual vector field $\theta^{\sharp}$ of $\theta$ satisfies
$$
\nabla_{\theta^{\sharp}}\xi=\gamma \theta^{\sharp},
$$
hence we have
\begin{equation}\label{k}
\theta^{\sharp}(|\xi|^2)=0,
\end{equation}
i.e., $|\xi|^2$ is constant along each integral curve of $\theta^{\sharp}$.
\end{proposition}

In \cite{bcd} we proved that
\begin{equation}\begin{aligned}\label{ghad}&
\frac{1}{2}\Delta(\theta(\grad(f)))=\gamma \div(\theta^{\sharp})-\beta\div(Q(\theta^{\sharp}))+\frac{\beta}{2}\theta^{\sharp}(\scal)\\&\hskip.3in+2\langle df,\Delta (\theta)\rangle-
\langle d(\Delta(f)), \theta\rangle+\delta (\nabla_{\grad(f)}\theta)(\grad(f))
\end{aligned}\end{equation}
and the following Bochner-type formula for a gradient generalized soliton.

\begin{proposition}\label{P:5.5} \cite{bcd}
Let $(M^n,g,\xi:=\grad(f),\beta,\gamma,\delta)$ be a gradient generalized soliton and $\theta$ a $1$-form on $M$. Then
\begin{equation}\begin{aligned}\label{gh}
&\left(\frac{1}{2}\Delta-\delta \nabla_{\grad(f)}\right)(\theta(\grad(f)))+\langle d(\Delta(f)), \theta\rangle-2\langle df,\Delta (\theta)\rangle-\gamma \div(\theta^{\sharp})\\&\hskip.3in +\beta\div(Q(\theta^{\sharp}))=
-\frac{\scal}{2}\theta^{\sharp}(\beta)+\frac{n}{2}\theta^{\sharp}(\gamma)+\frac{|\grad(f)|^2}{2}\theta^{\sharp}(\delta)\\&\hskip.5in+\frac{\delta}{2}\theta^{\sharp}(|\grad(f)|^2)-\frac{1}{2}\theta^{\sharp}(\Delta(f))-\delta \theta(\nabla_{\grad(f)}\grad(f)),\end{aligned}\end{equation}
where $\langle \theta_1,\theta_2\rangle:=\sum_{i=1}^n\theta_1(E_i)\theta_2(E_i)$, for any $1$-forms $\theta_1$ and $\theta_2$, and $\{E_i\}_{1\leq i\leq n}$ a local orthonormal frame field on $(M,g)$.
\end{proposition}

For $1$-forms orthogonal to $df$, from (\ref{ghad}) we deduce

\begin{proposition}\label{P:5.6}
Let $(M,g,\xi:=\grad(f),\beta, \gamma,\delta)$ be a gradient generalized soliton and let $\theta$ be a
closed and co-closed $1$-form on the closed Riemannian manifold $(M,g)$,
orthogonal to $\xi^{\flat}=df$. Then
$$\beta\div(Q(\theta^{\sharp}))= \frac{\beta}{2}\theta^{\sharp}(\scal)-\delta\theta(\nabla_{\grad(f)}\grad(f)).$$
\end{proposition}

As consequences, we have

\begin{corollary}\label{C:5.7}
Let $(M,g,\xi:=\grad(f),\gamma)$ be a gradient almost Ricci soliton and let $\theta$ be a
closed and co-closed $1$-form on the closed Riemannian manifold $(M,g)$,
orthogonal to $\xi^{\flat}=df$.

{\rm (i)} If $Q(\theta^{\sharp})$ is divergence-free, then
the scalar curvature of $M$ is constant along each integral curve of $\theta^{\sharp}$.

{\rm (ii)} In the compact case, $\int_M\theta^{\sharp}(\scal)dv=0$, where $dv$ is the volume element of $(M,g)$.
\end{corollary}

\begin{corollary}\label{C:5.8}
Let $(M,g,\xi:=\grad(f),\gamma,\delta)$ be a gradient almost $\xi^{\flat}$-Yamabe soliton and let $\theta$ be a
closed and co-closed $1$-form on the closed Riemannian manifold $(M,g)$,
orthogonal to $\xi^{\flat}=df$.
Then either the soliton is a gradient almost generalized Yamabe soliton or $\theta^{\sharp}$ is orthogonal to $\nabla_{\grad(f)}\grad(f)$.
\end{corollary}

\section{Generalized solitons on Euclidean submanifolds \\as examples}\label{Section6}

For general references on Euclidean submanifolds, we refer to \cite{book73,book20}. In this section, we will present some examples of generalized solitons on Euclidean submanifolds whose potential vector fields are given by their canonical vector fields.

For a Euclidean submanifold $M$ in $\mathbb E^{m}$ with position vector field $\x$, denote by $\x^T$ and $\x^N$ the tangential and normal components of  $\x$, respectively.
Thus, we have
\begin{align} \label{6.1} \x=\x^T+\x^N.\end{align}
The tangent vector field $\x^{T}$ on $M$ is called the {\it canonical vector field} of $M$.

Let $\phi:(M,g) \to (\mathbb E^m,\tilde g)$ be an isometric immersion of a Riemannian $n$-manifold $(M,g)$ into the Euclidean $m$-space $(\mathbb E^m,\tilde g)$.
We denote by $\nabla$ and $\widetilde\nabla$ the Levi-Civita connections of $(M,g)$ and of $(\mathbb E^m,\tilde g)$, respectively.
Then the formula of Gauss and the formula of Weingarten are given respectively by \begin{align} &\label{6.2}\widetilde \nabla_XY=\nabla_XY+h(X,Y),
\\& \label{6.3}\widetilde \nabla_X \zeta=-A_\zeta X+D_X\zeta,\end{align}
for vector fields $X,Y$ tangent to $M$ and $\zeta$ normal to $M$, where $h$ is the second fundamental form, $A$ the shape operator, and $D$ the normal connection of $\phi$.

The shape operator $A$ and the second fundamental form $h$ are related by
 \begin{align} &\nonumber\tilde g(h(X,Y),\zeta)=g(A_{\zeta}X,Y),\end{align}
 and the {\it mean curvature vector}  $H$ of $\phi$ is given by
   \begin{align}\nonumber H=\(\frac{1}{n}\){\rm Trace}\, (h).\end{align}

A submanifold $M$ is called {\it totally umbilical} if its second fundamental form $h$ satisfies
$ h(X,Y)=g(X,Y)H$, for any vector fields $X,Y$ tangent to $M$.

A hypersurface $M$ of $\mathbb E^{n+1}$ is called {\it quasi-umbilical} if its shape operator admits an eigenvalue, said $\kappa$, with multiplicity $mult(\kappa)\geq n-1$ (cf. \cite[page 147]{book73}). On the open subset $U$ of $M$ on which $mult(\kappa)=n-1$, an eigenvector with eigenvalue of multiplicity one is called a {\it distinguished principal direction} of the hypersurface $M$.

\begin{lemma} For a Euclidean submanifold $M$ in $\mathbb E^{m}$, we have
\begin{align} \label{6.6}& \nabla_X \x^T=A_{\x^N}X+X,
\\&\label{6.7} h(X, \x^T)=-D_X \x^N.\end{align}
\end{lemma}
\begin{proof} It is well-known that the position vector field $\x$ of $M$ in $\mathbb E^{m}$ is a concurrent vector field, i.e., $\x$ satisfies
\begin{align} \label{6.8} \widetilde\nabla_X\x=X,\end{align}
for any vector field $X$ tangent to $M$.
It follows from \e{6.1}, \e{6.2}, \e{6.3} and \e{6.8} that
\begin{equation}\begin{aligned} \label{6.9} X&=\widetilde \nabla_X \x^T+\widetilde\nabla_X \x^N=
 \nabla_X \x^T+h(X,\x^T)-A_{\x^N}X+D_X \x^N.
\end{aligned}\end{equation} After comparing the tangential and normal components from \e{6.9}, we obtain \e{6.6} and \e{6.7}.
\end{proof}

It follows from \e{6.6} and the definition of the Lie derivative that
\begin{equation}\begin{aligned} \nonumber({\mathcal L}_{\x^T}g)(X,Y)=2g(X,Y)+2g(A_{\x^N}X,Y),
\end{aligned}\end{equation} for any vector fields $X,Y$ tangent to $M$. After combining \e{generalsol} and \e{6.11}, we obtain the following.

\begin{theorem}\label{T:6.2} Let $M$ be a submanifold of $\mathbb E^{m}$. Then $(M,g,\x^T,\beta,\gamma,\delta)$ is a generalized soliton if and only if the Ricci tensor of $M$ satisfies
\begin{equation}\label{6.11}
\beta\,\Ric(X,Y)=(\gamma-1)g(X,Y)+\delta\, g(\x^{T},X)g(\x^{T},Y)-g(A_{\x^N}X,Y),
\end{equation} for any vector fields $X,Y$ tangent to $M$.
\end{theorem}

If the Euclidean submanifold $M$ lies in the unit hypersphere $S_{o}^{m-1}(1)$ of $\mathbb E^{m}$ centered at the origin $o\in\mathbb E^{m}$, then we have $\x^{N}=\x$ and $A_{\x^N}=-I$, where $I$ denotes the identity map. Thus, in this case, \e{6.11} reduces to
\begin{equation}\label{6.12} \beta\,\Ric(X,Y)=\gamma\,g(X,Y),\end{equation}
for any vector fields $X,Y$ tangent to $M$.

Consequently, we have the following.

\begin{corollary}\label{C:6.3} If a submanifold $M$ is contained in the unit hypersphere $S_{o}^{m-1}(1)$ of $\mathbb E^{m}$ centered at the origin $o\in \mathbb E^{m}$, then $(M,g,\x^T,\beta,\gamma,\delta)$ with $\beta$ nowhere zero on $M$ is a generalized soliton if and only if $(M,g)$ is an Einstein manifold. In this case, the ratio $\gamma:\beta$ is a constant.
\end{corollary}

If $\beta=0$, \e{6.12} also gives the following.

\begin{corollary}\label{C:6.4} Let $M$ be a submanifold lying in a hypersphere $S_{o}^{m-1}$ of $\mathbb E^{m}$ centered at the origin $o\in\mathbb E^{m}$. If $\beta=0$, then $(M,g,\x^T,\beta,\gamma,\delta)$ is a generalized soliton if and only if $\gamma=0$.\end{corollary}

If $M$ is a hypersurface of $\mathbb E^{m}$, then Theorem \ref{T:6.2} implies the following.

\begin{corollary}\label{C:6.5} Let $M$ be a hypersurface of $\mathbb E^{m}$ with $m>4$. If $(M,g,\x^T,\beta,\gamma,\delta)$ with $\delta = 0$ is a generalized soliton, then $M$ is a quasi-umbilical hypersurface. In particular, $M$ is a conformally flat space.
\end{corollary}

The next classification theorem was proved in \cite[Theorem 6.1]{cd14}.

\begin{theorem} \label{T:6.6} Let $(g,\x^T,\delta)$ define a Ricci soliton on a hypersurface $M^n$ of $\mathbb E^{n+1}$. Then $M$ is one of the following hypersurfaces of $\mathbb E^{n+1}$:
\\
{\rm (1)} a hyperplane through the origin;\\
{\rm (2)} a hypersphere centered at the origin;\\
{\rm (3)} an open part of a flat hypersurface generated by lines through the origin; \\
{\rm (4)} an open part of a circular hypercylinder $S^1(r)\times \mathbb E^{n-1}$, $r>0$;\\
{\rm (5)} an open part of a spherical hypercylinder $S^k(\sqrt{k-1})\times \mathbb E^{n-k}$, $2\leq k\leq n-1$.\end{theorem}

After combining this result with Theorem \ref{te44}, we have the following.

\begin{proposition} For each hypersurface $M$ of the 5 cases listed in Theorem \ref{T:6.6}, the dual $1$-form $(\x^{T})^{\flat}$ of the canonical vector field $\x^{T}$ of $M$ is  Schr\"{o}dinger-Ricci harmonic.
\end{proposition}

Recall that a vector field on a Riemannian manifold is called {\it conservative} if it is the gradient of some function, known as a {\it scalar potential}, and it is called {\it incompressible} if it is divergence-free.

The next result was obtained in \cite{chen17}.

\begin{theorem}\label{T:6.8} Let $M$ be a submanifold of $\mathbb{E}^m$.
Then we have:\\
{\rm (1)} the canonical vector field $\x^{T}$ of $M$ is always conservative;\\
{\rm (2)} $\x^{T}$ is incompressible if and only if $\<\x,H\>=-1$ holds identically on $M$.
\end{theorem}

Theorem \ref{T:6.8} implies the following.

\begin{corollary}\label{C:6.9} Let $M$ be a compact submanifold of $\mathbb E^{m}$. Then the dual $1$-form $(\x^{T})^{\flat}$ of the canonical vector field $\x^{T}$ of $M$ is harmonic if and only if $\<\x,H\>=-1$ holds identically on $M$.\end{corollary}

Several explicit examples of Euclidean submanifolds with $(\x^{T})^{\flat}$ harmonic $1$-form were given in \cite{chen17}. In particular, by applying \cite[Theorem 3.2]{chen17} and Corollary \ref{C:6.9} we obtain the following.

\begin{corollary} For every equivariantly isometrical immersion of a compact homogeneous Riemannian manifold $M$ into a Euclidean $m$-space $\mathbb E^{m}$, the dual $1$-form $(\x^{T})^{\flat}$ of the canonical vector field $\x^{T}$ of $M$ is a harmonic $1$-form.
\end{corollary}

\bigskip

\textit{Adara M. Blaga} 
%(corresponding author)

\textit{Department of Mathematics}

\textit{West University of Timi\c{s}oara}

\textit{Timi\c{s}oara, Rom\^{a}nia}

\textit{adarablaga@yahoo.com}

%}

\bigskip

\textit{Bang-Yen Chen}

\textit{Department of Mathematics}

\textit{Michigan State University}

\textit{East Lansing, MI, USA}

\textit{chenb@msu.edu}

%\bibitem{ca} G. Catino, Generalized quasi-Einstein manifolds with harmonic Weyl tensor, Math. Z. \textbf{271} (2012) 751--756.

%\bibitem{chen} \textbf{TO BE CITED !!!} B.-Y. Chen, Harmonic metrics, harmonic tensors and their applications, In book: Riemannian Geometry and Applications -- Proceedings RIGA 2021, Editura Univ. Bucure\c sti.

%\bibitem{cr} {M. Crasmareanu}, Last multipliers for multivectors with applications to Poisson geometry, Taiwanese J. Math. \textbf{13}(5) (2009) 1623--1636.

%\bibitem{cr2} {M. Crasmareanu}, Last multipliers for Riemannian geometries, Dirichlet forms and Markov diffusion semigroups, J. Geom. Anal. \textbf{27} (2017) 2618--2643.

%\bibitem{cras} {M. Crasmareanu}, Last multipliers on manifolds, Tensor \textbf{66}(1) (2005) 18--25.

%\bibitem{deng} {H. Deng}, Compact manifolds with positive $m$-Bakry-\'{E}mery-Ricci tensor, Differential Geom. Appl. \textbf{32} (2014) 88--96.

%\bibitem{ja} K.G.J. Jacobi, Vorlesungen \"{u}ber Dynamik, Berlin, 1866.

%\bibitem{vi} M. Vieira, Harmonic forms on manifolds with non-negative Bakry-\'{E}mery-Ricci curvature, Arch. Math. \textbf{101} (2013), 581--590.
\end{document}